\documentclass[a4paper,reqno,11pt]{amsart}
\usepackage[top=30mm,right=30mm,bottom=30mm,left=30mm]{geometry}

\usepackage{graphicx}
\usepackage{amsmath}
\usepackage{amssymb}
\usepackage{amsfonts}
\usepackage{amsthm}
\usepackage{amstext}
\usepackage{amsbsy}
\usepackage{amsopn}
\usepackage{amscd}
\usepackage{enumerate}
\usepackage{url}
\usepackage{color}
\usepackage[colorlinks]{hyperref}
\usepackage[hyperpageref]{backref}
\usepackage{rotating}
\usepackage{longtable}
\usepackage{float}
\usepackage{lscape}

 \usepackage{pgf,tikz}
\usetikzlibrary{arrows,shapes}
\tikzstyle{vertex}=[circle,inner sep=1pt,minimum size=5mm, font=\tiny]
\tikzstyle{edge style}=[line width=.5pt]
\tikzstyle{label style}=[circle,fill=white,inner sep=1pt,font=\tiny]
\tikzstyle{squirestyle}=[scale=.8,rounded corners,minimum height=5cm,minimum width=7cm,draw]
\tikzstyle{circlestyle}=[scale=.8,circle,minimum height=3cm, draw]
\tikzstyle{ellipsestyle}=[scale=.8,ellipse,dash pattern=on 5pt off 5pt,minimum height=2.3cm,minimum width=4.5cm, draw]
\tikzstyle{point}=[ circle, draw=black, inner sep=.3mm, font=\scriptsize]

\newtheorem{theorem}{Theorem}[section]
\newtheorem{corollary}[theorem]{Corollary}
\newtheorem{lemma}[theorem]{Lemma}
\newtheorem{proposition}[theorem]{Proposition}
\theoremstyle{definition}

\newtheorem{construction}[theorem]{Construction}
\newtheorem{remark}[theorem]{Remark}
\newtheorem{example}[theorem]{Example}
\newtheorem{problem}[theorem]{Problem}

\numberwithin{equation}{section}

\renewcommand{\S}{\mathrm{S}}

\newcommand{\Aut}{\mathbf{Aut}}

\newcommand{\Dmc}{\mathcal{D}}
\newcommand{\Pmc}{\mathcal{P}}
\newcommand{\Bmc}{\mathcal{B}}

\newcommand{\Omc}{\mathcal{O}}

\newcommand{\Rmc}{\mathcal{R}}
\newcommand{\Cmc}{\mathcal{C}}

\newcommand{\Qmc}{\mathcal{Q}}

\newcommand{\Oi}{\mathcal{O}_{\iota}}
\newcommand{\Oo}{\mathcal{O}_{o}}
\newcommand{\Ox}{\mathcal{O}_{X}}

\newcommand{\Kbf}{\mathbf{K}}

\newcommand{\n}{\mathbf{n}}

\renewcommand{\ni}{\mathbf{n}_{\iota}}
\newcommand{\no}{\mathbf{n}_{o}}
\newcommand{\nx}{\mathbf{n}_{X}}

\def\Sym{{\rm Sym}}

\def\la{\langle}
\def\ra{\rangle}
\def\Bhat{\widehat{\Bmc}}
\def\Dhat{\widehat{\Dmc}}

\renewcommand{\leq}{\leqslant}
\renewcommand{\geq}{\geqslant}

\newcommand{\imod}[1]{\allowbreak\mkern4mu({\operator@font mod}\,\,#1)}

\newcommand*\mathinhead[2]{\texorpdfstring{${#1}$}{#2}}

\begin{document}
 \title[Block-transitive two-designs based on grids]{Block-transitive two-designs\\ based on grids}

 \author[S.H. Alavi]{Seyed Hassan Alavi}%
 
 \address{Seyed Hassan Alavi, Department of Mathematics, Faculty of Science, Bu-Ali Sina University, Hamedan, Iran.
 }\email{alavi.s.hassan@basu.ac.ir}
 \author[A. Daneshkhah]{Ashraf Daneshkhah}%
 \thanks{Corresponding author: Ashraf Daneshkhah}
 \address{Ashraf Daneshkhah, Department of Mathematics, Faculty of Science, Bu-Ali Sina University, Hamedan, Iran.
 }%
  \email{adanesh@basu.ac.ir}
 \author{Alice Devillers}
 \address{ %
     Alice Devillers, Centre for the Mathematics of Symmetry and Computation, School of Mathematics and Statistics, The University of Western Australia, 35 Stirling Highway, Crawley, 6009 W.A.,
     Australia.}
 \email{alice.devillers@uwa.edu.au}
 \author{Cheryl E. Praeger}
 \address{ %
 Cheryl E. Praeger, Centre for the Mathematics of Symmetry and Computation, School of Mathematics and Statistics, The University of Western Australia, 35 Stirling Highway, Crawley, 6009 W.A.,
 Australia.}
 \email{cheryl.praeger@uwa.edu.au}
\thanks{The research was supported by Australian Research Council Discovery Grant DP200100080 of the third and fourth authors.}
 \date{June 24, 2021}%

\begin{abstract}
We study point-block incidence structures $(\Pmc,\Bmc)$ for which the point set $\Pmc$ is an $m\times n$ grid. Cameron and the fourth author showed that each block $B$ may be viewed as a subgraph of a complete bipartite graph $\Kbf_{m,n}$ with bipartite parts (biparts) of sizes $m, n$. In the case where $\Bmc$ consists of all the subgraphs isomorphic to $B$, under automorphisms of $\Kbf_{m,n}$ fixing the two biparts, they obtained necessary and sufficient conditions for $(\Pmc,\Bmc)$ to be a $2$-design, and to be a $3$-design. We first re-interpret these conditions more graph theoretically, and then focus on square grids, and designs admitting the full automorphism group of $\Kbf_{m,m}$. We find necessary and sufficient conditions, again in terms of graph theoretic parameters, for these incidence structures to be $t$-designs, for $t=2, 3$, and give infinite families of examples illustrating that block-transitive, point-primitive $2$-designs based on grids exist for all values of $m$, and flag-transitive, point-primitive examples occur for all even $m$. This approach also allows us to construct a small number of block-transitive $3$-designs based on grids.
\end{abstract}

\maketitle
\section{Introduction}\label{sec:intro}

A point-block incidence structure consists of a set $\Pmc$ of points, a set $\Bmc$ of blocks, and an incidence relation between $\Pmc$ and $\Bmc$. We will always take elements of $\Bmc$ to be subsets of $\Pmc$ with incidence as inclusion, and we assume that the point set $\Pmc$ is an $m\times n$ grid:
\begin{equation}\label{prc}
	\Pmc=\Rmc\times\Cmc,\quad\mbox{where $\Rmc, \Cmc$ are disjoint sets with $|\Rmc|=m, |\Cmc|=n$.}
\end{equation}
In their 1993 paper \cite[Section 3]{a:CamPr-btI-93}, Cameron and the fourth author studied incidence structures of this type, and viewed blocks in $\Bmc$ as subgraphs of an associated complete bipartite graph $\Kbf_{m,n}$ with vertex set $\Rmc\cup\Cmc$ and edges all pairs $\{R,C\}$ with $R\in\Rmc$ and $C\in\Cmc$. A $k$-element subset $B\subseteq \Pmc$ was viewed as the $k$-edge subgraph $\Delta$ of $\Kbf_{m,n}$ with edges the $k$ pairs $\{R,C\}$ such that the point $(R,C)$ of $\Pmc$ lies in $B$. Conversely, for a subgraph $\Delta$ of $\Kbf_{m,n}$, the associated subset of $\Pmc$ is the set $B(\Delta)$ of all pairs $(R,C)$ such that $R\in\Rmc$, $C\in\Cmc$, and $\{R,C\}$ is an edge of $\Delta$. The full automorphism group of $\Kbf_{m,n}$ is $K:=\Sym(\Rmc)\times\Sym(\Cmc)=\S_m\times\S_n$ if $m\ne n$, or $G=\S_m\wr\S_2$ (containing $K$ as an index $2$ subgroup) if $m=n$. The incidence structures analysed in \cite[Proposition 3.6]{a:CamPr-btI-93} are the following: for a given subgraph $\Delta$ of $\Kbf_{m,n}$,
\begin{equation}\label{dmc}
	\Dmc(\Delta)=(\Pmc,\Bmc)\quad\mbox{where $\Bmc =\{ B(\Delta^g) \mid g\in K\}$,}
\end{equation}
that is, the blocks are the subsets $B(\Delta')$, for all subgraphs $\Delta'$ of $\Kbf_{m,n}$ isomorphic to $\Delta$ under automorphisms from $K$. The group $K$ is, by definition, transitive on the block-set of $\Dmc(\Delta)$, and is also transitive on $\Pmc$ in its product action. Hence all blocks have the same size, namely the number $k$ of edges of $\Delta$, and each point $(R,C)$ lies in the same number of blocks.

For positive integers $t, v, k, \lambda$, an incidence structure $\Dmc=(\Pmc,\Bmc)$ is a \emph{$t$-$(v,k,\lambda)$ design}
if $v=|\Pmc|$, each block has size $k$, and each $t$-element subset of $\Pmc$ is contained in exactly $\lambda$ blocks.
Thus, taking $v=mn$, $\Dmc(\Delta)$ is a  $1$-$(v,k,\lambda)$ design for some $\lambda$, for any subgraph $\Delta$ with $k$ edges.
In this paper we build on work in \cite{a:CamPr-btI-93} to determine conditions for  $\Dmc(\Delta)$ to be a $t$-design for $t>1$ in terms of graph theoretic parameters of $\Delta$. The complement of a $t$-$(v,k,\lambda)$ design is also a $t$-design with block size $v-k$ and with the same automorphism group. Thus we assume that $k\leq v/2$. Moreover we are not interested in designs with block size $2$, since they are better interpreted as graphs, so we will assume that $3\leq k\leq v/2$.

The result \cite[Proposition 3.6]{a:CamPr-btI-93} gives necessary and sufficient conditions for $\Dmc(\Delta)$ to be a $2$-design, and to be a $3$-design, in terms of $m,n,k$ and certain other parameters associated with $\Delta$. We present this result with an additional graph theoretical interpretation of these conditions in Proposition~\ref{prop:K}, which assists us to study further the case of the square grid $m=n$. Note that $\Dmc(\Delta)$ is never a $4$-design (see  Lemma \ref{lem:not4}).

We then  focus on the case where $m=n$. Let $\Rmc = \{R_1,\dots,R_m\}$ and $\Cmc=\{C_1,\dots,C_m\}$. Here the graph $\Kbf_{m,m}$ admits the transposition map
\[
\tau:(R_i,C_j)\to (R_j, C_i),\ \mbox{for $i,j=1,\ldots,m$, and $\Aut(\Kbf_{m,m})=G=\la K,\tau\ra \cong\S_m\wr \S_2$.}
\]
In this case, we can define a second incidence structure using the larger group $G$ as follows.
\begin{equation}\label{dmchat}
	\Dhat(\Delta)=(\Pmc,\Bhat),\quad\mbox{where \ $\Bhat =\{ B(\Delta^g) \mid g\in G\}$.}
\end{equation}
The group $G$ is transitive on both $\Pmc$ and $\Bhat$, and hence $\Dhat(\Delta)$ is also a $1$-design, for any $\Delta$. Clearly the block set $\Bhat$ of $\Dhat(\Delta)$ contains $\Bmc$, and equality may or may not hold (see Lemma~\ref{lem:1design}(b)). Also $\Dhat(\Delta)$ may be a $t$-design for $t=2, 3$ (but not $t=4$ by Lemma~\ref{lem:not4}), and conceivably this can happen in three different ways, namely:
\begin{enumerate}
	\item[(1)] $ \Dhat(\Delta)=\Dmc(\Delta) $, and $\Dmc(\Delta)$  is a $t$-design;
	\item[(2)] $\Dhat(\Delta)\ne \Dmc(\Delta)$, and $\Dmc(\Delta)$ is a $t$-design (which implies that $\Dhat(\Delta)$ is also a $t$-design by \cite[Proposition 1.1]{a:CamPr-btI-93});
	\item[(3)] $\Dhat(\Delta)\ne \Dmc(\Delta)$, $\Dhat(\Delta)$ is a $t$-design, but $\Dmc(\Delta)$ is not a $t$-design.
\end{enumerate}
We consider these possibilities and obtain conditions on $\Delta$ and its parameters for obtaining $t$-designs.
Our main results for $m=n$ are Corollary~\ref{cor:K} (for $\Dmc(\Delta)$ to be a $t$-design),  Proposition~\ref{prop:G} (for Case (1)), Theorem~\ref{prop:G2} (for $\Dhat(\Delta)$ to be a $t$-design) and Corollary~\ref{cor:dhat} (for case (3)).  Given the results in \cite{a:CamPr-btI-93}, case (3) is perhaps the most interesting. In Example~\ref{ex1} we provide infinitely many $2$-designs for case (1) and also infinitely many $2$-designs for case (3) (see Lemma~\ref{lem:ex1}).  In the final Section~\ref{sec:3-design}, we construct some block-transitive $3$-designs and pose several open questions about existence of further examples.

Also, in Example~\ref{ex2}, we construct infinitely many flag-transitive $2$-designs for case (1) (see Lemma~\ref{lem:ex2}). Recall that a flag of an incidence structure is an incident point-block pair. We note that, in their study of flag-transitive $2$-designs with block-size four, Zhan, Zhou and Chen~\cite[Theorem 2]{a:ZZC-2018} showed that, if $\Delta$ has exactly four edges, and if $\Dhat(\Delta)$ is a flag-transitive $2$-design, then $m=5$, there are exactly two examples, and in each case $\Dhat(\Delta)=\Dmc(\Delta)$ (one has $\lambda=12$ and the other has $\lambda=18$). Neither of these designs lies in the family of flag-transitive $2$-designs constructed in Example~\ref{ex2}.

The designs $\Dmc(\Delta)$ and $\Dhat(\Delta)$ have also been studied recently in \cite{a:Braic-18}, seeking flag-transitive $1$-designs.  We comment on this study in Remark~\ref{rem:bmv}. There are other studies of flag-transitive $2$-designs relative to a group preserving a grid structure on the point set. For example, Cameron and the fourth author in \cite[Construction 7.2]{a:Praeger-Cameron-2014} show that, for each $r\geq2$, the symplectic design $S^-(r)$ admits a flag-transitive but point-imprimitive subgroup $2^{2r}\rtimes {\rm GL}(r,2)$ preserving a square $2^r\times 2^r$ grid structure on points. The full automorphism group $2^{2r}\rtimes {\rm Sp}(2r,2)$ of  $S^-(r)$ does not preserve this grid structure, but the fact that its flag-transitive subgroup  $2^{2r}\rtimes {\rm GL}(r,2)$ does shows that this symmetric design is a `subdesign' of one of the $2$-designs $\Dhat(\Delta)$ studied in this paper.  In \cite[Question 7.3]{a:Praeger-Cameron-2014} it was asked whether, for other symmetric designs,  there might exist flag-transitive subgroups of automorphisms which preserve a grid-structure on points. One such example arose in a recent study of point-imprimitive flag-transitive designs, where the third and fourth authors proved that there is a unique flag-transitive $2-(36,8,4)$-design \cite[Proposition 13]{a:DevillersPraeger2021}. Its full automorphism group, namely $S_6$, preserves a $6\times 6$ grid structure on the point set \cite[Construction 9 and Remark 12]{a:DevillersPraeger2021}. This design was also discovered independently by Zhang and Zhou~\cite[Theorem 1.3]{a:ZhangZhou2019} in their investigation of flag-transitive, point-quasiprimitive $2$-designs with $\lambda\leq 4$.

\begin{remark}\label{rem:bmv}
	In their 2018 paper~\cite{a:Braic-18}, Brai\'{c}, Mandi\'{c} and Vu\v{c}i\v{c}i\'{c}  study the subfamily of $1$-designs $\Dhat(\Delta)$ on which $G$ acts flag-transitively (that is, $G$ is transitive on the incident point-block pairs of $\Dhat(\Delta)$).
	They prove  in \cite[Proposition 4.1]{a:Braic-18} a version of \cite[Proposition 3.6]{a:CamPr-btI-93} for $G$-flag-transitive $1$-designs $\Dhat(\Delta)$: they characterise the property of $G$-flag-transitivity on $\Dhat(\Delta)$ in terms of a smaller incidence structure based on $B=B(\Delta)$, which they call $\Gamma(B)$. In particular they show that, for a $G$-flag-transitive $1$-design $\Dhat(\Delta)$, the group of automorphisms of $\Gamma(B)$ induced by the stabiliser $G_B$ is either flag-transitive, or `weakly flag-transitive' (that is, there are two equal length orbits on flags which are interchanged by a `duality' of $\Gamma(B)$).
	In \cite[Theorem 5.1] {a:Braic-18} they identify a subdivision of the family of $G$-flag-transitive $1$-designs $\Dhat(\Delta)$ into three types (different from, but reminiscent of our subdivision above), and they then proceed computationally, using MAGMA \cite{a:magma-97}, to find all flag-transitive  examples with $m<40$, and to determine whether or not such $1$-designs exist for $40\leq m\leq 63$. Unfortunately the interesting summary of the examples in \cite[Sections 7, 8]{a:Braic-18} does not provide any information on which of the examples are $t$-designs for $t\in\{ 2, 3\}$.
	
	Our work is independent of \cite{a:Braic-18}, principally because we are firstly interested  in finding when the  block-transitive incidence structures  $\Dhat(\Delta)$ and $\Dmc(\Delta)$ are $2$-designs or $3$-designs. Secondly, as mentioned above, we provide characterisations of these properties  in terms of graph theoretic conditions on $\Delta$. See Propositions~\ref{prop:K} and~\ref{prop:G}, Corollary~\ref{cor:K} and Theorem~\ref{prop:G2}.
	%
\end{remark}

\section{Parameters for the graph \mathinhead{\Delta}{Delta} and the block \mathinhead{B(\Delta)}{B(Delta)}}\label{sec:param}

We begin with a formal construction of the designs $\Dmc(\Delta)$ and $\Dhat(\Delta)$ starting from a subgraph $\Delta$, or equivalently, the corresponding point-subset $B(\Delta)$. In so doing we introduce the important associated parameters $x_i, y_j$.

\begin{construction}\label{con1}
	Let $m, n, k$ be positive integers with $3\leq k\leq mn/2$, let
	\[
	\mbox{$\Rmc = \{R_1,\dots,R_m\}$ and $\Cmc=\{C_1,\dots,C_n\}$ be disjoint sets, of size $m, n$ respectively},
	\]
	and let $\Kbf_{m,n}$ be the complete bipartite graph  with vertex set $\Rmc\cup \Cmc$, and edges all the pairs $\{R, C\}$ with $R\in\Rmc$ and $C\in\Cmc$. Also let $\Pmc=\Rmc\times\Cmc$.
	\begin{enumerate}[\rm (a)]
		\item Let $\Delta$ be a subgraph of $\Kbf_{m,n}$ having $k$ edges, and let
		\[
		B(\Delta)=\{(R,C)\in\Pmc \mid \mbox{$\{R,C\}$ is an edge of $\Delta$}\}.
		\]
		For $1\leq i\leq m$ and $1\leq j\leq n$, we define the parameters $x_i, y_j$ both in terms of $\Delta$ and of $B(\Delta)$, as follows:		
		\begin{align*}
			x_i&=\#\{j\mid \{R_i,C_j\}\ \mbox{is an edge of $\Delta$}\}\\
			&=\#\{j\mid  (R_i,C_j)\in B(\Delta)\},\\
			\mbox{and}\quad y_j&=\#\{i\mid \{R_i,C_j\}\ \mbox{is an edge of $\Delta$}\}\\
			&=\#\{i\mid  (R_i,C_j) \in B(\Delta)\}.
		\end{align*}
		
		\item Let $K=\Sym(\Rmc)\times\Sym(\Cmc)=\S_{m}\times \S_{n} $ acting as a group of automorphisms of $\Kbf_{m,n}$, and if $m=n$ also let $G=\langle K,\tau\rangle \cong\S_m\wr \S_2$, where
		\[
		\tau: (R_i,C_j)\to (R_j,C_i)\quad \mbox{for}\quad i,j=1,\dots, m.
		\]		
		\item Consider the product action of $K$, and also of $G$ if $m=n$, on $\Pmc$ and note that, for $g$ in $K$ or $G$, $\Delta^g$ is a $k$-edge subgraph of $\Kbf_{m,n}$ isomorphic to $\Delta$, and $B(\Delta^g)=(B(\Delta))^g$. Let
		\[
		\Bmc_K=\{ B(\Delta^g)\mid g\in K\}\quad\mbox{and}\quad \Bmc_G=\{ B(\Delta^g)\mid g\in G\},
		\]
		so that $\Dmc(\Delta)=(\Pmc,\Bmc_K)$, and if $m=n$, then $\Dhat(\Delta)= (\Pmc,\Bmc_G)$.
	\end{enumerate}
\end{construction}

We first record that this construction always produces block-transitive  $1$-designs.

\begin{lemma}\label{lem:1design}
	Let $\Kbf_{m,n}, \Delta, K, G, \tau$ be as in Construction~$\ref{con1}$.
	\begin{enumerate}[\rm (a)]
		\item $\Aut(\Kbf_{m,n})=K$ if $m\ne n$ and $\Aut(\Kbf_{m,n})=G$ if $m=n$;
		
		\item if $m=n$, then $G$ leaves $\Dmc(\Delta)$ invariant if and only if $\Delta^\tau=\Delta^x$ for some $x\in K$; and in this case $\Dhat(\Delta)=\Dmc(\Delta)$;
		
		\item $\Dmc(\Delta)$ is a $1$-$(mn,k,\lambda_1)$ design, for some $\lambda_1$, and $K$ is transitive on points and on blocks of $\Dmc(\Delta)$;
		
		\item if $m=n$ then $\Dhat(\Delta)$ is a $1$-$(m^2,k,\lambda_1')$ design, for some $\lambda_1'$, and $G$ is transitive on points and on blocks of $\Dhat(\Delta)$.
	\end{enumerate}
\end{lemma}

\begin{proof}
	Part (a) is well-known, and parts (c) and (d) follow from the fact that $K$ (resp. $G$) is point-transitive and block-transitive by construction. 
	For part (b), $G$ leaves   $\Dmc(\Delta)$ invariant if and only if $\tau$ does, and this, in turn, holds if and only if $\Delta^\tau\in\Bmc_K$, that is, $\Delta^\tau=\Delta^x$ for some $x\in K$.
\end{proof}

Next we prove that the designs $\Dhat(\Delta)$ and  $\Dmc(\Delta)$ cannot be $4$-designs. A permutation group on a set $X$ is said to be $t$-homogeneous if it is transitive on the  \emph{$t$-sets} ($t$-element subsets) of $X$.

\begin{lemma}\label{lem:not4}
	The designs $\Dmc(\Delta)$ and $\Dhat(\Delta)$ are not $4$-designs.
\end{lemma}
\begin{proof} Recall that we assume  $3\leq k\leq mn/2$.
	Suppose that $\Dmc=\Dmc(\Delta)$ or $\Dhat(\Delta)$ is a $4$-design.
	By definition, the group $K=S_m\times S_n$, or (if $m=n$) $G=S_m\wr S_2$, respectively, is a block-transitive group of automorphisms of  $\Dmc$, and neither of these groups is $2$-homogeneous on points.
	Thus, by \cite[Proposition 2.1(i)]{a:CamPr-btI-93}, it follows that $\Dmc$ is complete, that is all $k$-sets of points are blocks  (note that such a design is called trivial in \cite{a:CamPr-btI-93}). Hence the block-transitive group ($K$ or $G$) is $k$-homogeneous. Since $3\leq k\leq mn/2$, it follows from \cite[Theorem 9.4B]{b:Dixon} that the group is  $2$-transitive on points, which is a contradiction.
\end{proof}

Some properties of $\Dmc(\Delta)$ and $\Dhat(\Delta)$ depend on graph theoretic properties of $\Delta$.
A \emph{path} of length $u$, or \emph{$u$-path}, in a graph, sometimes denoted $P_u$,  is the subgraph induced by a sequence $(e_1,e_2,\dots,e_u)$ of $u$ pairwise distinct edges such that, for each $i<u$, $e_i$ and $e_{i+1}$ are incident with a common vertex, and $P_u$ has exactly $u+1$ vertices. Note that $P_u$ is also the subgraph induced by the sequence $(e_u,e_{u-1},\dots,e_1)$. 
In the bipartite graph $\Kbf_{m,n}$ all paths of odd length are equivalent under the group $K$, while there are two $K$-orbits on $u$-paths for $u$ even.
We say that a $2$-path $P$ is of \emph{type $\Rmc$} if $P$ involves one vertex of $\Rmc$, which we call its \emph{middle vertex}, and two vertices of $\Cmc$; and $P$ is of \emph{type $\Cmc$} otherwise.
A \emph{$3$-claw} in $\Delta$ is a subgraph isomorphic to $\Kbf_{1,3}$; it is said to have \emph{type $\Rmc$} if the vertex of valency $3$ lies in $\Rmc$, and to have \emph{type $\Cmc$} if
this vertex lies in $\Cmc$.
We record some graph theoretic quantities for $\Delta$ which can be expressed in terms of the parameters $x_i$ and $y_j$.

\begin{lemma}\label{lem:2-arcs}
	Let $\Delta, x_i,y_j$ be as in Construction~$\ref{con1}$.
	\begin{enumerate}[\rm (a)]
		\item The number of edges of $\Delta$ is $k=\sum_{i=1}^mx_i = \sum_{j=1}^n y_j$.
		
		\item The number of $2$-paths in $\Delta$ of type $\Rmc$ is $\sum_{i=1}^m \binom{x_i}{2}$;
		and the number of $2$-paths of type $\Cmc$ is $\sum_{j=1}^n\binom{y_j}{2}$.
		
		\item The number of $3$-claws in $\Delta$ of type $\Rmc$ is $\sum_{i=1}^m \binom{x_i}{3}$;
		and the number of $3$-claws of type $\Cmc$ is $\sum_{j=1}^n \binom{y_j}{3}$.
	\end{enumerate}
\end{lemma}

\begin{proof}
	Part (a) follows on noting that $x_i, y_j$ is the number of edges of $\Delta$ incident with $R_i$, $C_j$,
	respectively.
	
	For part (b), the number of $2$-paths with middle vertex $R_i$ is the number of ways to choose its pair of neighbours, that is $\binom{x_i}{2}$. Note that this number is zero if $x_i\leq 1$, as in this case no such pairs exist. It follows that the number of $2$-paths of type $\Rmc$ is  $\sum_{i=1}^m \binom{x_i}{2}$ as claimed. The number of $2$-paths of type $\Cmc$ is determined by a similar argument.
	
	An analogous argument proves part (c) on noting that each $3$-claw with $R_i$ the vertex of valency $3$ corresponds to a $3$-element subset of the $x_i$-element subset of vertices adjacent to $R_i$ (and a similar remark for $3$-claws of type $\Cmc$).
\end{proof}

In order to determine conditions for these designs to be $t$-designs for $t=2$ or $3$, we will use the following result referred to in \cite{a:CamPr-btI-93} as a `folklore result'.  For a positive integer $t$, we denote by $X^{\{t\}}$ the set of all $t$-sets of a set $X$.

\begin{proposition}[{\cite[Proposition 1.3]{a:CamPr-btI-93}}]\label{prop:CP}
	Let $G$ be a permutation group on a $v$-element set $\Pmc$, having orbits $\Omc_{1}$, \ldots $\Omc_{s}$ on $\Pmc^{\{t\}}$. Let $B\in \Pmc^{\{k\}}$, set $\Bmc:=\{B^g\mid g\in G\}$. Then for each $i=1,\ldots,s$, there is an integer $\n_i$ such that $\n_i=|(B')^{\{t\}}\cap\Omc_i|$ for each $B'\in \Bmc$. Moreover, the incidence structure $(\Pmc, \Bmc)$ is a $t$-design if and only if there exists a constant $c$ such that
	\[
	\frac{\n_1}{|\Omc_1|} = \frac{\n_2}{|\Omc_2|} = \dots = \frac{\n_s}{|\Omc_s|} = c.
	\]
	Moreover,  if $(\Pmc, \Bmc)$ is a $t$-$(v,k,\lambda)$ design, then
	\[
	c=\frac{\binom{k}{t}}{\binom{v}{t}}=\frac{\lambda}{|\Bmc|}.
	\]
	The group $G$ acts block-transitively on $(\Pmc,\Bmc)$, and $G$ is flag-transitive if and only if the setwise stabiliser $G_B$ of $B$ acts transitively on $B$.
\end{proposition}

We note that the first value of the constant $c$ follows from the equations $\n_i=c|\Omc_i|$ and the facts that $\sum_i\n_i=\binom{k}{t}$ and $\sum_i |\Omc_i|=\binom{v}{t}$. The second value follows from double-counting the number of $t$-sets and incident block pairs.

Proposition~\ref{prop:CP} has an immediate corollary for the flag-transitivity of the designs  $\Dmc(\Delta)$ and $\Dhat(\Delta)$, noting that the setwise stabilisers $K_{B(\Delta)}$ (respectively $G_{B(\Delta)}$) of $B(\Delta)$ in $K$
(resp. $G$) induce  subgroups of automorphisms of the graph $\Delta$.

\begin{corollary}\label{lem:flagtr}
	For $\Kbf_{m,n}, \Delta, K, G$ as in Construction~$\ref{con1}$, $K$ (resp. $G$) is transitive on the flags of $\Dmc(\Delta)$ (resp. $\Dhat(\Delta)$)  if and only if $K_{B(\Delta)}$ (resp. $G_{B(\Delta)}$) is transitive on the edges of $\Delta$.
\end{corollary}

We conclude this section with a technical remark on the work in \cite{a:Braic-18} which supplements Remark~\ref{rem:bmv}.

\begin{remark}
	As we mentioned in Remark~\ref{rem:bmv}, in the flag-transitive case, the authors of \cite{a:Braic-18} work with the incidence substructure induced on the sets of
	points and blocks incident to some edge of $\Delta$.  It follows from Corollary~\ref{lem:flagtr} that
	this substructure is a $K_{B(\Delta)}$-flag-transitive $1$-design if $K$ is flag-transitive on $\Dmc(\Delta)$, and is weakly $G_{B(\Delta)}$-flag-transitive if $m=n$ and $G$ is flag-transitive on $\Dhat(\Delta)$. The latter  situation is more complicated and is investigated in detail in \cite[Section 2]{a:Braic-18}. In our situation, where the actions are block-transitive, but not necessarily flag-transitive, we do not get a reduction to a smaller substructure with similar properties. Instead we  characterise the block-transitive designs by graph theoretic conditions on $\Delta$, and we explore such conditions in the next section.
\end{remark}

\section{Parametric conditions for \mathinhead{{t}}{t}-designs}\label{sec:tdesigns}

In a $1$-$(v,k,\lambda)$ design $\Dmc(\Pmc,\Bmc)$ we
denote the number of blocks by $b$, and the number of blocks containing a given point by $r$.
The quantities $v$, $b$, $r$, $k$, $\lambda$ are often referred to as the \emph{parameters} of $\Dmc$.
For the $1$-$(mn,k,\lambda_1)$ design $\Dmc(\Delta)$ and the $1$-$(m^2,k,\lambda_1')$ design $\Dhat(\Delta)$ introduced in Construction~\ref{con1}  we  also defined additional parameters $x_i, y_j$.
Here we explore conditions on these additional parameters, and on the graph $\Delta$, under which these incidence structures are $t$-designs for $t>1 $. Note that this means that $t$ is $2$ or $3$ by Lemma \ref{lem:not4}.

\subsection{The design \mathinhead{\Dmc(\Delta)}{mathcal{D}(Delta)}}\label{sub:dmc1}

In the case of the $K$-action on the design $\Dmc(\Delta)$,
appropriate conditions are given in \cite{a:CamPr-btI-93}, and we state them here, together with graph theoretic equivalents  and with a computation of the parameter $\lambda$.

\begin{proposition}[{\cite[Propositions 2.2 and 3.6]{a:CamPr-btI-93}}]\label{prop:K}
	Let $\Delta, K, \Dmc(\Delta), x_i, y_j$ be as in Construction~$\ref{con1}$. Then
	\begin{enumerate}[{\rm (a)}]
		\item $\Dmc(\Delta)$ is a $2$-design if and only if
		\begin{enumerate}[{\rm (i)}]
			\item $\displaystyle{\sum_{i=1}^m \binom{x_i}{2}=\frac{k(k-1)(n-1)}{2(mn-1)}=\ \mbox{number of $2$-paths in $\Delta$ of type $\Rmc$}}$; and
			\item $\displaystyle{\sum_{j=1}^n\binom{y_j}{2}=\frac{k(k-1)(m-1)}{2(mn-1)}=\ \mbox{number of $2$-paths in $\Delta$ of type $\Cmc$}}$.
		\end{enumerate}
		If these conditions hold, then  $\Dmc(\Delta)$ is a $2$-$(mn,k,\lambda)$ design with
		\[
		\lambda=\frac{k(k-1)(m-1)!(n-1)!}{(mn-1)|K_\Delta|}.
		\]
		\item $\Dmc(\Delta)$ is a $3$-design if and only if all the following hold:
		\begin{enumerate}[{\rm (i)}]
			\item the conditions in part (a) hold;
			\item the number of $3$-claws in $\Delta$ of type $\Rmc$ is
			\[
			\sum_{i=1}^m \binom{x_i}{3}=\frac{k(k-1)(k-2)(n-1)(n-2)}{6(mn-1)(mn-2)};
			\]
			\item the number of $3$-claws in $\Delta$ of type $\Cmc$ is
			\[
			\sum_{j=1}^n \binom{y_j}{3}=\frac{k(k-1)(k-2)(m-1)(m-2)}{6(mn-1)(mn-2)};
			\]
			\item the number of $3$-paths  in $\Delta$ is $\displaystyle{\frac{k(k-1)(k-2)(m-1)(n-1)}{(mn-1)(mn-2)}}$.
		\end{enumerate}
		If these conditions hold, then  $\Dmc(\Delta)$ is a $3$-$(mn,k,\lambda)$ design with
		\[
		\lambda=\frac{k(k-1)(k-2)(m-1)!(n-1)!}{(mn-1)(mn-2)|K_\Delta|}.
		\]
	\end{enumerate}
\end{proposition}
\begin{proof}
	We refer to \cite[Proposition 3.6]{a:CamPr-btI-93} for the conditions to be a $2$- or $3$-design. (We note that no details are given of the proof of  but its proof is very similar to the proof we give of Lemma \ref{prop:G2}.) The graph theoretic conditions (number of $2$-paths of type $\Rmc$, etc.) follow Lemma~\ref{lem:2-arcs}.
	
	Now we compute the parameter $\lambda$ in each case,  using that $\lambda=\frac{\binom{k}{t}}{\binom{v}{t}}b$, where $b=|\Bmc|$, by Proposition \ref{prop:CP}.
	Since $K$ is block-transitive, we have that
	\[
	b=\frac{|K|}{|K_\Delta|}=\frac{m!n!}{|K_\Delta|}.
	\]
	
	Assume first $\Dmc(\Delta)$ is a $2$-design. Then
	\[
	\lambda= \frac{k(k-1)}{mn(mn-1)}\cdot b.
	\]
	Substituting for $b$ from the displayed equation above we get $\lambda$ as in the statement.
	Assume now $\Dmc(\Delta)$ is a $3$-design. Then
	\[
	\lambda= \frac{k(k-1)(k-2)}{mn(mn-1)(mn-2)}\cdot b.
	\]
	Substituting for $b$ from the displayed equation above we get $\lambda$ as in the statement.
\end{proof}

The conditions in this result simplify when $m=n$ and it is worthwhile stating them.

\begin{corollary}\label{cor:K}
	Let $\Delta, K, \Dmc(\Delta), x_i, y_j$ be as in Construction~$\ref{con1}$ with $m=n$. Then
	\begin{enumerate}[{\rm (a)}]
		\item $\Dmc(\Delta)$ is a $2$-design if and only if the number of $2$-paths in $\Delta$ of type $\Rmc$ is equal to the number of $2$-paths in $\Delta$ of type $\Cmc$, and this number is
		\[
		\sum_{i=1}^m \binom{x_i}{2}=\sum_{j=1}^m \binom{y_j}{2}=\frac{k(k-1)}{2(m+1)}.
		\]
		If these conditions hold, then $\Dmc(\Delta)$ is a $2$-$(m^2,k,\lambda)$ design with
		\begin{equation}\label{e:lambda-2}
			\lambda=\frac{k(k-1)(m-1)!(m-2)!}{(m+1)|K_\Delta|}.
		\end{equation}
		
		\item $\Dmc(\Delta)$ is a $3$-design if and only if the conditions in part (a) hold and also the following two  additional conditions hold:
		\begin{enumerate}[{\rm (i)}]
			\item the number of $3$-claws in $\Delta$ of type $\Rmc$ is equal to the number of $3$-claws in $\Delta$ of type $\Cmc$, and this number is
			\[
			\sum_{i=1}^m \binom{x_i}{3}=\sum_{j=1}^m \binom{y_j}{3}=
			\frac{k(k-1)(k-2)(m-2)}{6(m+1)(m^2-2)};
			\]
			\item the number of $3$-paths in $\Delta$ is $\displaystyle{\frac{k(k-1)(k-2)(m-1)}{(m+1)(m^2-2)}}$.
			
		\end{enumerate}
		If these conditions hold, then  $\Dmc(\Delta)$ is a $3$-$(m^2,k,\lambda)$ design with 	
		\begin{equation}\label{e:lambda-3}
			\lambda=\frac{k(k-1)(k-2)(m-1)!(m-2)!}{(m+1)(m^2-2)|K_\Delta|}.
		\end{equation}
	\end{enumerate}
\end{corollary}

\subsection{The designs \mathinhead{\Dmc(\Delta)}{mathcal{D}(Delta)} and \mathinhead{\Dhat(\Delta)}{dh} when \mathinhead{\Dmc(\Delta)}{mathcal{D}(Delta)} is a $2$-design}\label{sub:dmc2}

When $m=n$, and in the special case where $\Delta$ has the property that the map $\tau$ in Construction~\ref{con1} induces an isomorphism from $\Delta$ to $\Delta^x$ for some $x\in K$, the group $G$ acts as a group of automorphisms of $\Dmc(\Delta)$, by Lemma~\ref{lem:flagtr}(b), and in this case the two designs $\Dmc(\Delta)$ and $\Dhat(\Delta)$ are equal. Moreover $\Aut(\Delta)$ is the setwise stabiliser $G_{B(\Delta)}$ with $B(\Delta)$ as in  Construction~\ref{con1}(a), and in this case  $\Aut(\Delta)$ contains $K_{B(\Delta)}$ as a subgroup of index $2$ and interchanges the parts $\Rmc$ and $\Cmc$ of the bipartition of $\Kbf_{m,m}$. As a consequence the number of $2$-paths in $\Delta$ of type $\Rmc$ and of type $\Cmc$ are equal, and the number of $3$-claws in $\Delta$ of type $\Rmc$ and of type $\Cmc$ are equal. This leads to a further simplification of the conditions in Corollary~\ref{cor:K}.

\begin{proposition}\label{prop:G}
	Let $\Delta, K, \tau, \Dmc(\Delta), \Dhat(\Delta), x_i, y_j$ be as in Construction~$\ref{con1}$, and suppose that  $m=n$ and that $\Delta^\tau=\Delta^x$ for some $x\in K$, so that $\Dhat(\Delta)=\Dmc(\Delta)$. Then
	\begin{enumerate}[{\rm (a)}]
		\item $\Dmc(\Delta)$ is a $2$-design if and only if the number of $2$-paths in $\Delta$ is
		\[
		\sum_{i=1}^m x_i(x_i-1) = \sum_{j=1}^m y_j(y_j-1) = \frac{k(k-1)}{m+1},
		\]
		and if this is the case then $\Dmc(\Delta)$ is a $2$-$(m^2,k,\lambda)$ design with
		$\lambda$ as in \eqref{e:lambda-2}.
		\item $\Dmc(\Delta)$ is a $3$-design if and only if the conditions in part (a) hold and also the following two  additional conditions hold:
		\begin{enumerate}[{\rm (i)}]
			\item the number of $3$-claws in $\Delta$ is
			\[
			2\sum_{i=1}^m \binom{x_i}{3}=2\sum_{j=1}^m \binom{y_j}{3}=
			\frac{k(k-1)(k-2)(m-2)}{3(m+1)(m^2-2)};
			\]
			\item the number of $3$-paths  in $\Delta$ is $\displaystyle{\frac{k(k-1)(k-2)(m-1)}{(m+1)(m^2-2)}}$.
		\end{enumerate}
		If these conditions hold, then  $\Dmc(\Delta)$ is a $3$-$(m^2,k,\lambda)$ design with 	$\lambda$ as in \eqref{e:lambda-3}.
	\end{enumerate}
\end{proposition}

\begin{remark}
	On the other hand, if $m=n$, and if the graph $\Delta^\tau$  is not the image of $\Delta$ under any element of $K$, then the incidence structure $\Dhat(\Delta)$ has twice as many blocks as $\Dmc(\Delta)$. We note that this condition that `$\Delta^\tau\ne \Delta^x$ for any $x\in K$' is equivalent to the condition  $\Aut(\Delta)=K_{B(\Delta)}$.
	Moreover, the condition that $\Dmc(\Delta)$ is a $t$-design implies that also $\Dhat(\Delta)$ is a $t$-design, for $t\in\{2,3\}$ (see \cite[ Proposition~1.1]{a:CamPr-btI-93}). Thus the conditions given in Corollary~\ref{cor:K} are \emph{sufficient, but not necessarily necessary} for $\Dhat(\Delta)$ to be a $t$-design, for $t=2, 3$.
\end{remark}

\subsection{The design \mathinhead{\Dhat(\Delta)}{dh} when \mathinhead{\Dmc(\Delta)}{\mathcal{D}(Delta)} is not necessarily a $2$-design} \label{sub:dhat}

Here we consider the case where $m=n$, and where $\Dmc(\Delta)$ is not necessarily a $2$-design.
We use Proposition~\ref{prop:CP} to determine conditions for $\Dhat(\Delta)$ to be a $t$-design, where $t=2, 3$. We clarify in Remark~\ref{rem:dhat} how to decide whether or not $\Dmc(\Delta)$ is a $t$-design for $t=2$ or $3$ from these conditions.

\begin{theorem}\label{prop:G2}
	Let $\Delta, G, \tau, \Dhat(\Delta), x_i, y_j$ be as in Construction~$\ref{con1}$, and suppose that  $m=n$. Then
	\begin{enumerate}[{\rm (a)}]
		\item $\Dhat(\Delta)$ is a $2$-design if and only if the number of $2$-paths in $\Delta$ is $k(k-1)/(m+1)$, or equivalently,
		\[
		\sum_{i=1}^m \binom{x_i}{2}+\sum_{j=1}^n \binom{y_j}{2}=\frac{k(k-1)}{m+1}.
		\]
		If this condition holds, then  $\Dhat(\Delta)$ is a $2$-$(m^2,k,\lambda)$ design with
		\[
		\lambda=\frac{2k(k-1)(m-1)!(m-2)!}{(m+1) |G_\Delta|}.
		\]
		
		\item $\Dhat(\Delta)$ is a $3$-design if and only if the condition in part (a) holds and also the following two  additional conditions hold:
		\begin{enumerate}[{\rm (i)}]
			\item the number of $3$-claws in $\Delta$ is
			\[
			\sum_{i=1}^m\binom{x_i}{3}+\sum_{j=1}^m \binom{y_j}{3}=
			\frac{k(k-1)(k-2)(m-2)}{3(m+1)(m^2-2)};
			\]
			\item the number of $3$-paths in $\Delta$ is $\displaystyle{\frac{k(k-1)(k-2)(m-1)}{(m+1)(m^2-2)}}$.
		\end{enumerate}
		If these conditions hold, then  $\Dhat(\Delta)$ is a $3$-$(m^2,k,\lambda)$ design with
		\[
		\lambda=\frac{2k(k-1)(k-2)(m-1)!(m-2)!}{(m+1)(m^2-2) |G_\Delta|}.
		\]
	\end{enumerate}
\end{theorem}

\begin{proof}
	We apply Proposition~\ref{prop:CP}. First we consider necessary and sufficient conditions for $\Dhat(\Delta)$ to be a $2$-design.
	The group $G$ has just $2$ orbits on unordered pairs from $\Pmc$, namely the set
	\[
	\Oo=\{\{(R, C),(R', C')\}\mid R, R'\in \Rmc,\ C, C'\in\Cmc, \text{ and } R\neq R', C\neq C' \},
	\]
	and its complement $\Oi=\Pmc^{\{2\}}\setminus \Oo$. Note that $|\Oi|=m^2(m-1)$ and $|\Oi|+|\Oo|= |\Pmc^{\{2\}}|
	= m^2(m^2-1)/2$.  For $B=B(\Delta)$, and $X\in\{ \iota, o\}$, let
	$\nx=|\Ox\cap B^{\{2\}}|$ (note this number does not depend on the choice of $B$ in the block-set of $\Dhat(\Delta)$). Then by Proposition~\ref{prop:CP}, $\Dhat(\Delta)$ is a $2$-design if and only if $\ni/|\Oi|=\no/|\Oo|$, that is, $\ni|\Oo|=\no|\Oi|$. Adding $\ni|\Oi|$ to each side, and noting that $\ni+\no=\binom{k}{2}$, we see that this holds if and only if
	\[
	\ni \binom{m^2}{2} = \ni (|\Oi| + |\Oo|) = (\ni+\no) |\Oi| = \binom{k}{2} m^2(m-1),
	\]
	which is equivalent to the condition $\ni= k(k-1)/(m+1)$. Thus we have proved that $\Dhat(\Delta)$ is a $2$-design if and only if $\ni= k(k-1)/(m+1)$.
	
	Now a pair $\{(R, C), (R,C')\}\in \Oi\cap B^{\{2\}}$ corresponds to a  $2$-path with vertex set $\{C,R,C'\}$ in $\Delta$ of type $\Rmc$, and similarly a pair $\{(R, C), (R',C)\}\in \Oi\cap B^{\{2\}}$ corresponds to a $2$-path in $\Delta$ of type $\Cmc$. These correspondences are one-to-one, and so $ \ni$ is the number of $2$-paths in $\Delta$, and by Lemma~\ref{lem:2-arcs}, this number is
	$\sum_{i=1}^m \binom{x_i}{2}+\sum_{j=1}^n \binom{y_j}{2}$. Thus the design criterion in part (a) is proved.
	
	Now we consider conditions under which $\Dhat(\Delta)$ is a $3$-design, again using the criterion in Proposition~\ref{prop:CP}. There are four $G$-orbits in $\Pmc^{\{3\}}$ (so we expect three conditions on the parameters). These $G$-orbits are as follows (including a description of a representative $3$-set $\{(R,C), (R',C'), (R'', C'')\}\subseteq\Pmc$ in the orbit).
	\[
	\begin{array}{l|l|l}\hline
		\mbox{Orbit} & \mbox{Description} & \mbox{Size of orbit}\\ \hline
		\Qmc_1	&	R=R'=R'' \mbox{ and $C, C', C''$ pairwise distinct,} 		& m^2(m-1)(m-2)/3\\
		&  \mbox{or $R, R', R''$ pairwise distinct }	\mbox{and }C=C'=C''		&	\\
		\Qmc_2 &	R=R'\ne R'' \mbox{ and } C=C''\ne C'	& m^2(m-1)^2\\
		\Qmc_3 &	R=R'\ne R'' \mbox{ and } C, C',C'' \mbox{ pairwise distinct,}& m^2(m-1)^2(m-2)\\
		&	\mbox{or }R, R', R'' \mbox{ pairwise distinct and }\ C=C'\ne C''	&	\\
		\Qmc_4 &	R, R', R'' \mbox{ pairwise distinct and } 	& m^2(m-1)^2(m-2)^2/6\\
		& 	C, C', C'' \mbox{ pairwise distinct }		&	\\ \hline
	\end{array}
	\]
	For $i=1.\dots,4$, let $\n_i=|\Qmc_i\cap B^{\{3\}}|$ (which does not depend on the choice of block $B$). Then by Proposition~\ref{prop:CP}, $\Dhat(\Delta)$ is a $3$-design if and only if
	\[
	\frac{\n_1}{|\Qmc_1|} = \frac{\n_2}{|\Qmc_2|} = \frac{\n_3}{|\Qmc_3|} = \frac{\n_4}{|\Qmc_4|} = \frac{\binom{k}{3}}{\binom{m^2}{3}},
	\]
	and we note that, the equation
	\begin{equation}\label{eq:ni}
		\frac{\n_i}{|\Qmc_i|} = \frac{\binom{k}{3}}{\binom{m^2}{3}}
	\end{equation}
	holding for any three values of
	$i\in\{1, 2, 3, 4\}$ implies that this equation also holds for the fourth value of $i$.
	For $i=1$, using the value for $|\Qmc_1|$,  equation \eqref{eq:ni} simplifies to
	\[
	\n_1=\frac{k(k-1)(k-2)(m-2)}{3(m+1)(m^2-2)}.
	\]
	In this case there is a one-to-one correspondence between  $\Qmc_1\cap B^{\{3\}}$ and the set of $3$-claws in $\Delta$, namely a $3$-set  $\{(R,C), (R,C'), (R, C'')\}\in\Qmc_1\cap B^{\{3\}}$ corresponds to the $3$-claw $\Kbf_{1.3}$ of type $\Rmc$ with vertex set $\{R, C, C', C''\}$, and the number of the $3$-claws of type $\Rmc$ in $\Delta$ is $\sum_{i=1}^m\binom{x_i}{3}$. Similar comments apply to the $3$-sets in $\Qmc_1\cap B^{\{3\}}$ involving a single $C\in\Cmc$, and the $3$-claws in $\Delta$ of type $\Cmc$.
	Thus $\n_1$ is equal to the number of $3$-claws in $\Delta$ and this, in turn, is equal to
	$\sum_{i=1}^m\binom{x_i}{3} + \sum_{j=1}^m\binom{y_j}{3}$. Thus  equation \eqref{eq:ni} for $i=1$ is equivalent to condition (b)(i).
	
	For $i=2$, using the value for $|\Qmc_2|$,  equation \eqref{eq:ni} simplifies to
	\[
	\n_2=\frac{k(k-1)(k-2)(m-1)}{(m+1)(m^2-2)}.
	\]
	In this case there is a one-to-one correspondence between  $\Qmc_2\cap B^{\{3\}}$ and the set of $3$-paths in $\Delta$, namely a $3$-set  $\{(R,C), (R,C'), (R'', C)\}\in\Qmc_2\cap B^{\{3\}}$ corresponds to the $3$-path with vertex set  $\{C',R,C,R''\}$ in $\Delta$.
	Thus $\n_2$ is equal to the number of $3$-paths in $\Delta$ and hence  equation \eqref{eq:ni} for $i=2$ is equivalent to condition (b)(ii).
	
	For $i=3$, using the value for $|\Qmc_3|$,  equation \eqref{eq:ni} simplifies to
	\begin{equation}\label{eq:n3}
		\n_3=\frac{k(k-1)(k-2)(m-1)(m-2) }{(m+1)(m^2-2)}.
	\end{equation}
	At this stage we have proved that: \emph{given that (b)(i) and (b)(ii) hold, the incidence structure
		$\Dhat(\Delta)$ is a $3$-design if and only if \eqref{eq:n3} holds.}
	
	Let $X$ be a $3$-subset of $\Pmc$. Then, from the description in the table above, $X\in\Qmc_3$
	if and only if the induced subgraph $[X]$ of $\Delta$ is a vertex-disjoint union $P_2 + \Kbf_2$ of a $2$-path $P_2$ and an edge $\Kbf_2$.  In order to count such $3$-subsets $X$, let us define the set
	\[
	\mathcal{T}:=\{ (P, E)\mid P \mbox{ is a $2$-path  of $\Delta$, and $E$ is an edge of $\Delta$ not in $P$}\}.
	\]
	Let $Y$ be the number of $2$-paths of $\Delta$. For each $2$-path $P$ there are $k-2$ choices for an edge $E$ such that $(P,E)\in\mathcal{T}$. Hence $|\mathcal{T}|=Y(k-2)$. Now, for each  $(P, E)\in\mathcal{T}$, the subgraph generated by the three edges in $P
	\cup E$ is one of (1) a $3$-claw, (2) a $3$-path, or (3) $P_2+\Kbf_2$.
	We know, from  (b)~(i) and (b)~(ii), the numbers of $3$-claws  and $3$-paths  in $\Delta$. Moreover,  each  $3$-claw  contains three $2$-paths, and each $3$-path  contains two $2$-paths.  Hence the number $N$ of pairs $(P, E)$ which correspond to a subgraph of type (1) or (2) is
	\[
	N=3\times \frac{k(k-1)(k-2)(m-2)}{3(m+1)(m^2-2)}  + 2\times \frac{k(k-1)(k-2)(m-1)}{(m+1)(m^2-2)}.
	\]
	Thus the number of pairs $(P, E)\in\mathcal{T}$ such that $P\cup E$ generates a subgraph $P_2+\Kbf_2$ is
	\[
	|\mathcal{T}|-N = (k-2)\left( Y - \frac{k(k-1)(m-2)}{(m+1)(m^2-2)} -  \frac{2k(k-1)(m-1)}{(m+1)(m^2-2)}\right),
	\]
	and as discussed above, this number is equal to $|\Qmc_3\cap B^{\{3\}}|=\n_3$. Thus \eqref{eq:n3} holds if and only if
	\begin{align*}
		Y 	&= \frac{k(k-1)(m-1)(m-2) }{(m+1)(m^2-2)} + \frac{k(k-1)(m-2)}{(m+1)(m^2-2)} +  \frac{2k(k-1)(m-1)}{(m+1)(m^2-2)}\\
		&=\frac{k(k-1)}{(m+1)(m^2-2)}\left( (m-1)(m-2) + (m-2) + 2(m-1)  \right)\\
		&= \frac{k(k-1)}{(m+1)(m^2-2)}\left( m^2-2\right) = \frac{k(k-1)}{m+1}.
	\end{align*}
	It follows that  \eqref{eq:n3} holds if and only if the number of $2$-paths in $\Delta$ is  $k(k-1)/(m+1)$, that is to say, if and only if part (a) holds. This completes the proof of the design criteria in part (b).

	Now we compute the parameter $\lambda$ in part (a) and (b) under the assumption that the design criteria hold, and using the equality $\lambda=b\binom{k}{t}/\binom{v}{t}$ from Proposition \ref{prop:CP}.
	Since $G$ is block-transitive we have that
	\[
	b=\frac{|G|}{|G_\Delta|}=\frac{2m!^2}{|G_\Delta|}.
	\]
	Assume first $\Dmc(\Delta)$ is a $2$-design. Then
	\[
	\lambda= \frac{k(k-1)}{m^2(m^2-1)}\cdot b=\frac{k(k-1)}{m^2(m+1)(m-1)}\cdot b.
	\]
	Substituting for $b$ from the displayed equation above we get $\lambda$ as in the statement.
	Assume now $\Dmc(\Delta)$ is a $3$-design. Then
	\[
	\lambda= \frac{k(k-1)(k-2)}{m^2(m^2-1)(m^2-2)}\cdot b .
	\]
	Substituting for $b$ from the displayed equation above we get $\lambda$ as in the statement.
\end{proof}

%
%


\begin{remark}\label{rem:dhat}
	In Theorem~\ref{prop:G2} we derived conditions on the parameters $x_i, y_j$ and the graph $\Delta$ under which $\Dhat(\Delta)$ is a $t$-design, for $t=2, 3$. These conditions do not depend on whether or not $\Dmc(\Delta)$ is a $t$-design. To find conditions \emph{for $\Dhat(\Delta)$ to be a $t$-design {and at the same time} $\Dmc(\Delta)$ not to be a $t$-design}, we  compare Theorem~\ref{prop:G2} and Corollary~\ref{cor:K} and obtain the following corollary.
\end{remark}
	\begin{corollary}\label{cor:dhat}
		Let $\Delta, G, \tau, \Dmc(\Delta), \Dhat(\Delta), x_i, y_j$ be as in Construction~$\ref{con1}$ with  $m=n$. Then $\Dhat(\Delta)$ is a $t$-design but $\Dmc(\Delta)$ is not a $t$-design if and only if:
		\begin{enumerate}[{\rm (a)}]
			\item for $t=2$: the condition of Theorem~$\ref{prop:G2}(a)$ holds, and \[\sum_{i=1}^m\binom{x_i}2\ne \frac{k(k-1)}{2(m+1)};\]		
			\item for $t=3$: the conditions of Theorem~$\ref{prop:G2}(b)$ hold, and also either
			\[\sum_{i=1}^m\binom{x_i}2\ne \frac{k(k-1)}{2(m+1)},\quad or\quad \sum_{i=1}^m\binom{x_i}3 \ne  \frac{k(k-1)(k-2)(m-2)}{6(m+1)(m^2-2)}.\]
		\end{enumerate}
	\end{corollary}	
	
	For $t=2$, we give in the next section an infinite family of examples $\Delta$ for which both $\Dhat(\Delta)$ and $\Dmc(\Delta)$ are $2$-designs, and an infinite family of examples where $\Dhat(\Delta)$ is a $2$-design but $\Dmc(\Delta)$ is not (Example~\ref{ex1} and Lemma~\ref{lem:ex1}).

\section{Examples of $2$-designs}\label{sec:2designs}

We construct several infinite families of $k$-edged subgraphs $\Delta$ of $\Kbf_{m,m}$, taking the vertex set of $\Kbf_{m,m}$ to be $\Rmc\cup\Cmc$, where $\Rmc=\{R_1,\dots, R_m\}$ and $\Cmc=\{C_1,\dots, C_m\}$.

\begin{example}\label{ex1}
	Let $k$ be an integer such that $k\geq 3$, and let $\Delta=P_k$, a $k$-path with edges $\{R_i,C_i\}$ for $1\leq i\leq \lceil k/2\rceil$, and $\{R_{i+1},C_i\}$ for $1\leq i\leq \lfloor k/2\rfloor$. Thus the parameters from Construction~\ref{con1} are the following for $k=2a+1$:
	\[
	x_i= \begin{cases}
		1	&\text{if }i=1			\\
		2	&\text{if }2\leq i\leq a+1	\\
		0 	&\text{otherwise}
	\end{cases}
	\]		
	\[
	y_j=	\begin{cases}
		1	&\text{if }j=a+1			\\
		2	&\text{if }1\leq j\leq a	\\
		0 	&\text{otherwise}
	\end{cases}
	\]
	and for $k=2a$ they are:
	\[
	x_i= \begin{cases}
		1	&\text{if }i=1, a+1				\\
		2	&\text{if }2\leq i\leq a	\\
		0 	&\text{otherwise}
	\end{cases}
	\]		
	\[
	y_j=	\begin{cases}
		2	&\text{if }1\leq j\leq a	\\
		0 	&\text{otherwise.}
	\end{cases}
	\]
\end{example}

\begin{lemma}\label{lem:ex1}
	Let $\Delta=P_k$ as in Example~$\ref{ex1}$. Then $\Dhat(\Delta)$ is not a $3$-design,  and $\Dhat(\Delta)$ is a $2$-design if and only if $m=k-1$. Moreover,  if $m=k-1$, then
	\begin{enumerate}[\rm (a)]
		\item if $k$ is odd then
		\begin{enumerate}[\rm (i)]
			\item $\Dhat(\Delta)=\Dmc(\Delta)$, and hence $\Dmc(\Delta)$ is a $2$-design, and
			\item $\Dhat(\Delta)$ is a $2$-$((k-1)^2, k, \lambda)$ design with
			\[
			\lambda = \frac{(k-1)!(k-3)!}{((k-3)/2)!^2};
			\]
		\end{enumerate}
		\item if $k$ is even then
		\begin{enumerate}[\rm (i)]
			\item $\Dhat(\Delta)\ne\Dmc(\Delta)$ and $\Dmc(\Delta)$ is not a $2$-design, and
			\item $\Dhat(\Delta)$ is a $2$-$((k-1)^2, k, \lambda)$ design with
			\[
			\lambda = \frac{(k-1)!(k-3)!}{(k/2-2)!(k/2-1)!}.
			\]
		\end{enumerate}
	\end{enumerate}
\end{lemma}

\begin{proof}
	Note that neither $\Dmc(\Delta)$ nor $\Dhat(\Delta)$ is a $3$-design since $\Delta$ does not contain any $3$-claws (see Corollary~\ref{cor:K} and Theorem~\ref{prop:G2}). Also note that the number of $2$-paths in $\Delta$ is $k-1$. If $k$ is odd then $(k-1)/2$ of them have type $\Rmc$ and $(k-1)/2$ have type $\Cmc$, while if $k$ is even,  $k/2$ of them have type $\Rmc$ and $(k-2)/2$ have type $\Cmc$.
	
	By Theorem~\ref{prop:G2}(a), $\Dhat(\Delta)$ is a $2$-design if and only if the number of $2$-paths in $\Delta$ is $k(k-1)/(m+1)$. Since this number is $k-1$ it follows that  $\Dhat(\Delta)$ is a $2$-design if and only if  $m=k-1$.
	For the rest of the proof we may therefore assume that $m=k-1$, so that the number of points is $m^2=(k-1)^2$ and $\Dhat(\Delta)$ is a $2$-design.
	By Theorem \ref{prop:G2}, we can compute the parameter $\lambda$:
	\begin{align*}
		\lambda&=\frac{2k(k-1)(m-1)!(m-2)!}{(m+1) |G_\Delta|}\\&=\frac{2k(k-1)(k-2)!(k-3)!}{k |G_\Delta|}\\&=\frac{2(k-1)!(k-3)!}{ |G_\Delta|}
	\end{align*}

	Assume first that $k=2a+1$ is odd, so $m=2a$.
	It is not difficult to check that the $k$-path $\Delta^\tau$ is the image of $\Delta$ under an element of $K$, with $K, \tau$ as in Construction~\ref{con1}, and hence $\Dhat(\Delta)=\Dmc(\Delta)$, by Lemma~\ref{lem:flagtr}(b). In this case  the stabiliser of the `block' $\Delta$ in $G= S_m\wr S_2$ is $G_\Delta=(S_{a-1}\times S_{a-1})\cdot C_2$, which has order $2(a-1)!^2=2\left(\frac{k-3}2\right)!^2$. The stated value for $\lambda$ follows.
	
	Assume now that $k=2a$ is even, with $a\geq2$, so $m=2a-1$.
	Here the condition of Corollary~\ref{cor:K}(a) fails, so $\Dmc(\Delta)$ is not a $2$-design, and in particular $\Dhat(\Delta)\ne \Dmc(\Delta)$. In this case  the stabiliser of the `block' $\Delta$ in $G= S_m\wr S_2$ is $G_\Delta=(S_{a-2}\times S_{a-1})\times C_2$, which has order $2(a-2)!(a-1)!=2\left(\frac{k-4}2\right)!\left(\frac{k-2}2\right)!$.
	The stated value for $\lambda$ follows.
\end{proof}

Note that the $2$-design  $\Dhat(\Delta)$ in Lemma~\ref{ex1} is complete when $k=3$ and $m=2$.
Also, when $k\geq 4$, we have $k=m+1\leq m^2/2$. Thus
Lemma~\ref{lem:ex1} has the following immediate corollary, noting that the group $S_m\wr S_2$ 
acts primitively in its product action of degree $m^2$ when $m\geq3$.

\begin{corollary}\label{cor:ex1}
	For each $m\geq3$, there exists a $2$-$(m^2, m+1,\lambda)$ design (for some $\lambda$), admitting $S_m\wr S_2$ as a block-transitive, point-primitive group of automorphisms.
\end{corollary}

By Lemma~\ref{lem:flagtr}, the designs in Example~\ref{ex1} with $k\geq3$ are not $G$-flag-transitive since $\Aut(P_k)$ is not edge-transitive for $k\geq 3$. We modify the example to produce a family of flag-transitive designs.

\begin{example}\label{ex2}
	Let $k=2a\geq4$, an even integer, and let $\Delta=C_k$, a cycle of length $k$ with edges $\{R_i,C_i\}$  for $1\leq i\leq a$, $\{R_{i+1},C_i\}$ for $1\leq i\leq a-1$, and $\{R_1,C_a\}$.
	The parameters from Construction~\ref{con1} are the following: $x_i=2$ for $1\leq i\leq a$ and $x_i=0$ otherwise; and $y_j=2$ for $1\leq j\leq a$ and $y_j=0$ otherwise.
\end{example}

\begin{lemma}\label{lem:ex2}
	Let $\Delta=C_k$ with $k$ even, as in Example~$\ref{ex2}$. Then $\Dmc(\Delta)=\Dhat(\Delta)=\Dmc$, say,
	is a $1$-design, and  the groups $K$ and $G$ of Construction~$\ref{con1}$ are both flag-transitive on $\Dmc$. Moreover, $\Dmc$ is never a $3$-design, and $\Dmc$ is a $2$-design if and only if $m=k-2$, in which case $\Dmc$ is a $2$-$((k-2)^2, k, \lambda)$ design with $\lambda= \frac{(k-3)!(k-4)!}{(k/2-2)!^2}$.
\end{lemma}
\begin{proof}
	Both $\Dmc(\Delta)$ and $\Dhat(\Delta)$ are 1-designs by Lemma~\ref{lem:flagtr} parts (c) and (d), and since  the stabilisers $K_B$ and $G_B$ of $B=B(\Delta)$ are both transitive on the  set of edges of $\Delta$, it follows that the groups $K$ and $G$ act flag-transitively on $\Dmc(\Delta)$ and $\Dhat(\Delta)$, respectively, by Lemma~\ref{lem:flagtr}(e).
	Moreover, the image   $\Delta^\tau$ of $\Delta$ under  the map $\tau$ defined in Construction~\ref{con1}  is equal to the image of $\Delta$ under an element of $K$, and hence $\Dhat(\Delta)=\Dmc(\Delta)$, by Lemma~\ref{lem:flagtr}(b). By Corollary~\ref{cor:K}, $\Dmc(\Delta)$ is never a $3$-design since $\Delta$ contains no $3$-claws.
		
	The parameters $x_i, y_j$ in Construction~$\ref{con1}$ for $\Delta$ satisfy $\sum_{i=1}^m\binom{x_i}2=\sum_{j=1}^m\binom{y_j}2 = k/2$. Hence the numbers of  $2$-paths  in $\Delta$ of types $\Rmc$ and $\Cmc$ are equal to each other, and each is equal to $k/2$ (Lemma~\ref{lem:2-arcs}). This implies that $\Dmc(\Delta)=\Dhat(\Delta)$ is a $2$-design if and only if $k-1=m+1$, that is $m=k-2$ (Corollary~\ref{cor:K} and/or Theorem~\ref{prop:G2}).
	
	Assume now that $m=k-2$, so the number of points is $(k-2)^2$.
	By Theorem \ref{prop:G2}, we can compute the parameter $\lambda$:
	\begin{align*}
		\lambda&=\frac{2k(k-1)(m-1)!(m-2)!}{(m+1) |G_\Delta|}\\&=\frac{2k(k-1)(k-3)!(k-4)!}{(k-1) |G_\Delta|}\\&=\frac{2k(k-3)!(k-4)!}{ |G_\Delta|}
	\end{align*}
	As $k$ is even, let $k=2a$ so $m=2a-2$. Then  the stabiliser of the `block' $\Delta$ in $G= S_m\wr S_2$ is $G_\Delta=((S_{a-2}\times S_{a-2})\times D_{2a}).2$,  which has order $4a(a-2)!^2=2k(k/2-2)!^2$.
	The stated value for $\lambda$ follows.
\end{proof}

Note the design  $\Dhat(\Delta)$ is complete when $k=4$ and $m=2$. Also, when $k\geq 6$, we have  $k=m+2\leq m^2/2$, and in this case the group $G=\Sym(m)\wr\Sym(2)$ is point-primitive.

\begin{corollary}\label{cor:ex2}
	For each even $m\geq4$, there exists a $2$-$(m^2, m+2,\lambda)$ design (for some $\lambda$) admitting
	$\Sym(m)\wr\Sym(2)$ as a flag-transitive, point-primitive group of automorphisms.
\end{corollary}

\section{Examples of $3$-designs}\label{sec:3-design}

In our efforts in Section~\ref{sec:2designs} to construct explicit infinite families of designs using Construction~\ref{con1}, it turned out that our families contained many $2$-designs but no $3$-designs. Unfortunately
we have not succeeded in finding infinitely many $3$-designs from the construction.

\begin{problem}
	Find an infinite family of graphs $\Delta$ such that $\Dmc(\Delta)$ or $\Dhat(\Delta)$ is a $3$-design, or prove that no such infinite family exists.
\end{problem}

Indeed it would be very interesting to have an infinite family of graphs $\Delta$ such that $\Dmc(\Delta)$ or $\Dhat(\Delta)$ are $3$-designs.    We did manage to construct explicitly a small number of individual $3$-designs using Construction~\ref{con1}, and we describe both how we searched for them, and the designs themselves.

By \cite[Proposition 1.1]{a:CamPr-btI-93} it follows that, if $\Dmc(\Delta)$ is a $3$-design then $\Dhat(\Delta)$ is also a $3$-design. Thus we decided first to search for examples where $\Dhat(\Delta)$ is a $3$-design with $m^2$ points and block size $k$.
By Theorem \ref{prop:G2} the following divisibility conditions must hold:
\begin{enumerate}[\rm (1)]
	\item $m+1$ divides $k(k-1)$;
	\item $3(m+1)(m^2-2)$ divides $k(k-1)(k-2)(m-2)$; and
	\item $(m+1)(m^2-2)$ divides $k(k-1)(k-2)(m-1)$.
\end{enumerate}

Conditions (1)--(3) are quite restrictive. For instance for $m$ up to $100$ and $3\leq k\leq m^2/2$, the only possibilities for $[m,k]$ are
\begin{equation}\label{eq:mk}
	[ 11, 36 ],\ [ 25, 91 ],\ [ 38, 105 ],\ [ 41, 805 ],\ [ 54, 1365 ],\ [ 74, 2025 ],\ [ 87, 2256 ].
\end{equation}
\begin{problem}
	Decide whether or not there exist any $3$-designs $\Dhat(\Delta)$ arising from Construction~\ref{con1} with $[m,k]$ one of the pairs in \eqref{eq:mk}, and if such exist, then classify them.
\end{problem}

\subsection{Some $3$-$(121,36,\lambda)$ designs $\Dhat(\Delta)$}

We studied the smallest case, where $m=11$ and $k=36$, so $\Delta$ is a subgraph of $\Kbf_{11,11}$ with $k=36$ edges. We found several examples of subgraphs
$\Delta$ that satisfy all the conditions of Theorem \ref{prop:G2}(b) and hence yield $3$-designs $\Dhat(\Delta)$. We have not classified all such designs.

\begin{figure}
	\begin{center}
		\begin{tikzpicture}[scale=0.52]
			\node (v2) at (0,0) [point] {3};
			\node (v5) at (1,0) [point] {2};
			\node (v7) at (2,0) [point] {2};
			\node (v1) at (3,0) [point] {1};
			\node (v9) at (4,0) [point] {1};
			\node (v10) at (5,0) [point] {1};
			\node (v11) at (6,0) [point] {1};
			
			\node (v13) at (0,2) [point] {3};
			\node (v16) at (1,2) [point] {2};
			\node (v19) at (2,2) [point] {2};
			\node (v12) at (3,2) [point] {1};
			\node (v18) at (4,2) [point] {1};
			\node (v21) at (5,2) [point] {1};
			\node (v22) at (6,2) [point] {1};
			\draw[thick] (v1) -- (v13) -- (v2)--(v12)--(v5)--(v16)--(v1);
			\draw[thick] (v16) -- (v1) -- (v18)--(v7)--(v19)--(v9)--(v21)--(v1);
			\draw[thick] (v10) -- (v21);
			\draw[thick] (v11) -- (v12);
		\end{tikzpicture}
		\caption{A subgraph $\Delta$ of $\Kbf_{11,11}$ yielding a $3$-design $\Dhat(\Delta)$.}
		\label{fig:3design}
	\end{center}
\end{figure}

We make a few comments about our search: suppose that $\Delta$ is a subgraph of $\Kbf_{11,11}$ with $k=36$ edges, and let $E(\Delta)$ denote the set of edges. Let $x_i, y_j$ be the parameters for $\Delta$ defined as in Construction~\ref{con1}. To decide whether or not $\Delta$ yields a $3$-design $\Dhat(\Delta)$, we need to determine the number of various subgraphs of $\Delta$, namely, the numbers of $2$-paths, $3$-paths, and $3$-claws. For example, to count the number of $3$-paths of $\Delta$, we note that an edge $\{R_i,C_j\}$ of $\Delta$ can be the centre edge of a $3$-path as long as $x_i>1$ and $y_j>1$, and hence the number of $3$-paths of $\Delta$ is
\[
\sum_{\{ R_i,C_j\}\in E(\Delta)} (x_i-1)(y_j-1).
\]
By Theorem~\ref{prop:G2}(b)(ii), this number must be $300$.
The diagram in Figure~\ref{fig:3design} represents a subgraph $\Delta$ satisfying all the properties of Theorem~\ref{prop:G2}(b). It is one of several we found, and is the one that has the most symmetry. Each circle represents a certain number of vertices (written inside). An edge between two circles means each pair of vertices (one from each circle) is joined by an edge, and no edge between two circles means there are no edges between the vertices in the two circles.
From the description of this graph we see that the stabiliser $G_\Delta$ (which is the automorphism group of $\Delta$) is $(S_3\times S_2\times S_2)^2$ of order $(24)^2$, and hence
by Theorem \ref{prop:G2}, we can compute the parameter $\lambda$:
\begin{align*}
	\lambda&=\frac{2k(k-1)(k-2)(m-1)!(m-2)!}{(m+1)(m^2-2) |G_\Delta|}\\
	&=\frac{72\cdot 35\cdot 34\cdot 10!\cdot 9!}{12\cdot  119\cdot (24)^2}\\&= \frac{(10!)^2}{96}=137,168,640,000.
\end{align*}

For this graph $\Delta$ it is not difficult to see that $\Delta^\tau\ne \Delta^k$ for any $k\in K$, so $\Dhat(\Delta)\neq \Dmc(\Delta)$ (by Lemma~\ref{fig:3design}(b)). Moreover $\sum_{i=1}^{11} \binom{x_i}{2}\neq \sum_{j=1}^{11} \binom{y_j}{2}$ so $\Dmc(\Delta)$ is not even a $2$-design (by Corollary~\ref{cor:K}(a)), even though $\Dhat(\Delta)$ is a $3$-design.

\subsection{Designs $\Dmc(\Delta)$ with $m\ne n$.}

We also looked for $3$-designs $\Dmc(\Delta)$ in the general case where $m,n$ are not necessarily equal. Proposition \ref{prop:K}(b) yields five divisibility conditions involving the parameters $m, n, k$.
The smallest possible value of $m$, $n$ for which all five of these divisibility conditions hold,  occurs for  $m=8,n=2,k=6$.
The diagram in Figure~\ref{fig:1design} represents a graph $\Delta$ which yields the unique example of a $3$-design with these parameters, that is, it is the only graph $\Delta$ satisfying all the conditions of Proposition \ref{prop:K}(b).
\begin{figure}
	\begin{center}
		\begin{tikzpicture}[scale=0.52]
			\node (v1) at (0,0) [point] {1};
			\node (v2) at (1,0) [point] {3};
			\node (v3) at (2,0) [point] {1};
			\node (v4) at (3,0) [point] {3};

			\node (v5) at (0,2) [point] {1};
			\node (v6) at (1,2) [point] {1};
			
			\draw[thick] (v1) -- (v5) -- (v2);
			\draw[thick] (v1) -- (v6) -- (v3);
		\end{tikzpicture}
		\caption{A subgraph $\Delta$ of $\Kbf_{8,2}$ yielding a $3$-$(16,6,80)$ design $\Dmc(\Delta)$. }
		\label{fig:1design}
	\end{center}
\end{figure}

The next  smallest possibility for $m$, $n$ where the five divisibility conditions all hold occurs for $m=11,n=7,k=20$. An example of a subgraph $\Delta$ with these parameters yielding a $3$-design $\Dmc(\Delta)$ is given in  \cite[p.39]{a:CamPr-btI-93}.

\subsection{Designs $\Dmc(\Delta)$ with $m= n$.}
We also looked for $3$-designs $\Dmc(\Delta)$  where $m=n$ (so $\Dhat(\Delta)$ is also a $3$-design). Corollary \ref{cor:K} yields five divisibility conditions involving the parameters $m, k$.
The smallest possible value of $m$, $k$ for which all five of these divisibility conditions hold,  occurs for  $m=38,k=105$.
The diagram in Figure~\ref{fig:1design} represents a graph $\Delta$ which yields an example of a $3$-design with these parameters.

From the description of this graph we see that the stabiliser $K_\Delta=G_\Delta$  is $(S_3\times S_2\times S_2)^2$ of order $2^9\cdot 3!\cdot 4!\cdot 5!^2$, and hence
by  Corollary \ref{cor:K}, we can compute the parameter $\lambda$ for $\Dmc(\Delta)$:
\begin{align*}
	\lambda&=\frac{k(k-1)(k-2)(m-1)!(m-2)!}{(m+1)(m^2-2) |K_\Delta|}\\
	&=\frac{105\cdot 104\cdot 103\cdot 37!\cdot 36!}{39\cdot  (38^2-2)\cdot 2^9\cdot 3!\cdot 4!\cdot 5!^2}\\
	&\approx 9.6\cdot 10^{76}.
\end{align*}

For this graph $\Delta$ it is not difficult to see that $\Delta^\tau\ne \Delta^k$ for any $k\in K$, so $\Dhat(\Delta)\neq \Dmc(\Delta)$ (by Lemma~\ref{fig:3design}(b)).
By Theorem \ref{prop:G2}, since $G_\Delta=K_\Delta$, the parameter $\lambda$ for  $\Dhat(\Delta)$ is twice the parameter $\lambda$ for $\Dmc(\Delta)$ computed above.
\begin{figure}
	\begin{center}
		\begin{tikzpicture}[scale=0.52]
			\node (v1) at (1,0) [point] {1};
			\node (v2) at (2,0) [point] {1};
			\node (v3) at (3,0) [point] {1};
			\node (v4) at (4,0) [point] {1};
			\node (v5) at (5,0) [point] {1};
			\node (v6) at (6,0) [point] {1};
			\node (v7) at (7,0) [point] {1};
			\node (v8) at (8,0) [point] {1};
			\node (v11) at (10,0) [point] {1};
			\node (v12) at (11,0) [point] {1};
			\node (v13) at (12,0) [point] {1};
			\node (v16) at (14,0) [point] {1};
			\node (v19) at (16,0) [point] {1};
			\node (v22) at (18,0) [point] {1};
			\node (v25) at (20,0) [point] {1};
			\node (v29) at (22,0) [point] {1};
			\node (v30) at (23,0) [point] {1};
			\node (v31) at (24,0) [point] {1};
			\node (v9) at (9,0) [point] {2};
			\node (v14) at (13,0) [point] {2};
			\node (v17) at (15,0) [point] {2};
			\node (v20) at (17,0) [point] {2};
			\node (v23) at (19,0) [point] {2};
			\node (v32) at (25,0) [point] {2};
			\node (v26) at (21,0) [point] {3};
			\node (v34) at (26,0) [point] {5};
			
			\node (v41) at (2,3) [point] {1};
			\node (v43) at (3,3) [point] {1};
			\node (v45) at (4,3) [point] {1};
			\node (v46) at (5,3) [point] {1};
			\node (v47) at (6,3) [point] {1};
			\node (v48) at (7,3) [point] {1};
			\node (v49) at (8,3) [point] {1};
			\node (v50) at (9,3) [point] {1};
			\node (v51) at (10,3) [point] {1};
			\node (v52) at (11,3) [point] {1};
			\node (v53) at (12,3) [point] {1};
			\node (v54) at (13,3) [point] {1};
			\node (v55) at (14,3) [point] {1};
			\node (v56) at (15,3) [point] {1};
			\node (v57) at (16,3) [point] {1};
			\node (v58) at (17,3) [point] {1};
			\node (v59) at (18,3) [point] {1};
			\node (v60) at (19,3) [point] {1};
			\node (v63) at (21,3) [point] {1};
			\node (v66) at (23,3) [point] {1};
			\node (v69) at (25,3) [point] {1};
			\node (v70) at (26,3) [point] {1};
			\node (v71) at (27,3) [point] {1};
			\node (v61) at (20,3) [point] {2};
			\node (v64) at (22,3) [point] {2};
			\node (v67) at (24,3) [point] {2};
			\node (v39) at (1,3) [point] {4};
			\node (v72) at (28,3) [point] {5};
			
			\draw[thick] (v39) -- (v1) -- (v41);
			\draw[thick] (v45) -- (v1) -- (v43);
			\draw[thick] (v1) -- (v46);
			\draw[thick] (v39) -- (v2) -- (v41);
			\draw[thick] (v2) -- (v45);
			
			\draw[thick] (v39) -- (v3) -- (v41);
			\draw[thick] (v41) -- (v4) -- (v45);
			\draw[thick] (v46) -- (v4) -- (v47);
			\draw[thick] (v41) -- (v5) -- (v46);
			\draw[thick] (v48) -- (v5) -- (v47);
			\draw[thick] (v46) -- (v6) -- (v47);
			\draw[thick] (v48) -- (v6) -- (v49);
			
			\draw[thick] (v46) -- (v7) -- (v47);
			\draw[thick] (v48) -- (v7);
			
			\draw[thick] (v46) -- (v8) -- (v49);
			\draw[thick] (v50) -- (v8);
			\draw[thick] (v51) -- (v9) -- (v52);
			\draw[thick] (v50) -- (v9);
			\draw[thick] (v51) -- (v11) -- (v52);
			\draw[thick] (v53) -- (v11);
			\draw[thick] (v53) -- (v12) -- (v52);
			\draw[thick] (v54) -- (v12);
			\draw[thick] (v53) -- (v13) -- (v55);
			\draw[thick] (v54) -- (v13);
			\draw[thick] (v54) -- (v14) -- (v55);
			\draw[thick] (v56) -- (v14);
			\draw[thick] (v57) -- (v16) -- (v58);
			\draw[thick] (v54) -- (v16);
			\draw[thick] (v58) -- (v17) -- (v59);
			\draw[thick] (v54) -- (v17);
			\draw[thick] (v59) -- (v19) -- (v60);
			\draw[thick] (v54) -- (v19);
			\draw[thick] (v61) -- (v20) -- (v60);
			\draw[thick] (v61) -- (v22) -- (v63);
			\draw[thick] (v63) -- (v23) -- (v64);
			\draw[thick] (v64) -- (v25) -- (v66);
			
			\draw[thick] (v67) -- (v26) -- (v69);
			\draw[thick] (v69) -- (v29) -- (v70);
			\draw[thick] (v66) -- (v30) -- (v70);
			\draw[thick] (v70) -- (v31) -- (v71);
			\draw[thick] (v71) -- (v32);
		\end{tikzpicture}
		\caption{A subgraph $\Delta$ of $\Kbf_{38,38}$ yielding $3$-designs $\Dmc(\Delta)$ and $\Dhat(\Delta)$.}
		\label{figb:3design}
	\end{center}
\end{figure}


\end{document}